\documentclass[11pt]{amsart}%
\usepackage{graphicx,amscd, color,amsmath,amsfonts,amssymb,geometry,amssymb}
\usepackage{amsmath}
\usepackage{amsfonts}
\usepackage{amssymb}
\usepackage{graphicx}%
\providecommand{\U}[1]{\protect\rule{.1in}{.1in}}
\providecommand{\U}[1]{\protect\rule{.1in}{.1in}}
\newtheorem{theorem}{Theorem}[section]

\newtheorem{corollary}[theorem]{Corollary}

\newtheorem{problem}[theorem]{Problem}

\numberwithin{equation}{section}

\begin{document}
\title{Sharp coincidences for absolutely summing multilinear operators}
\author[D. Pellegrino ]{Daniel Pellegrino}
\address{Departamento de Matem\'{a}tica \\
Universidade Federal da Para\'{\i}ba \\
58.051-900 - Jo\~{a}o Pessoa, Brazil.}
\email{pellegrino@pq.cnpq.br and dmpellegrino@gmail.com}
\thanks{\textsuperscript{**}This note is part of the author's thesis which is being
written for the candidature to the position of Full Professor at Universidade
Federal da Para\'{\i}ba, Brazil.}

\begin{abstract}
In this note we prove the optimality of a family of known coincidence theorems
for absolutely summing multilinear operators. We connect our results with the
theory of multiple summing multilinear operators and prove the sharpness of
similar results obtained via the complex interpolation method.

\end{abstract}
\maketitle

\section{Preliminaries and background}

A long standing problem raised by Banach in \cite[page 40]{Banach32} (see also
the Problem 122 of the Scottish Book \cite{Mau}) asks whether in every
infinite-dimensional Banach space there exists an unconditionally convergent
series which fails to be absolutely convergent. The positive solution, in
1950, due to A. Dvoretzky and C. A. Rogers \cite{DR}, is probably the main
motivation for the appearance of the concept of absolutely summing operators
in the 1950-1960's with the works of A. Grothendieck \cite{grothendieck}, A
Pietsch \cite{piet} and J. Lindenstrauss and A. Pe\l czy\'{n}ski \cite{LP}.

Essentially, if $E$ and $F$ are Banach spaces, an absolutely summing operator
$u:E\rightarrow F$ is a linear operator that improves the convergence of
series in the following fashion: each unconditionally summable sequence
$\left(  x_{n}\right)  _{n=1}^{\infty}$ in $E$ is sent to an absolutely
summable sequence $\left(  u(x_{n})\right)  _{n=1}^{\infty}$ in $F$. More
generally, if $1\leq p\leq q<\infty,$ a continuous linear operator
$u:E\rightarrow F$ is absolutely $(q;p)$-summing if $\left(  u(x_{j})\right)
_{j=1}^{\infty}\in\ell_{q}(F)$ whenever
\[
\sup_{\varphi\in B_{E^{\ast}}}%
{\displaystyle\sum\limits_{j=1}^{\infty}}
\left\vert \varphi(x_{j})\right\vert ^{p}<\infty,
\]
where $E^{\ast}$ denotes the topological dual of $E$ and $B_{E^{\ast}}$
represents its closed unit ball (for the theory of absolutely summing
operators we refer to \cite{Di} and, for recent results, \cite{z, dies} and
references therein).

The space of all absolutely $(q;p)$-summing operators from $E$ to $F$ is
denoted by $\Pi_{q;p}(E;F)$ (or $\Pi_{p}(E;F)$ if $p=q$). It is not difficult
to prove that if $1\leq p\leq q<\infty$, then $\Pi_{p}\subset$ $\Pi_{q}.$
Henceforth the space of all bounded linear operators from a Banach space $E$
to a Banach space $F$ will be represented by $\mathcal{L}(E;F)$.

The theory of absolutely summing operators is nowadays a mandatory topic in
modern Banach Space Theory (see \cite{AK, grande, LT, Ry, W}), with somewhat
unexpected applications. For example, using tools of the theory of absolutely
summing operators we can prove that if $E=\ell_{1}$ or $E=c_{0}$ every
normalized unconditional basis is equivalent to the unit vector basis of $E$
(see \cite{LP, LT}). According to Pietsch \cite[page 365]{PH}, one of the most
profound results in Banach Space Theory is Grothendieck's \textit{theor\`{e}me
fondamental de la th\'{e}orie m\'{e}trique des produits tensoriels}, from the
famous Grothendieck's R\'{e}sum\'{e} \cite{grothendieck} (see also \cite{re}
for a modern approach), which can be rewritten in the language of absolutely
summing operators as follows:

\begin{theorem}
[Grothendieck]Every bounded linear operator from $\ell_{1}$ to any Hilbert
space is absolutely summing.
\end{theorem}

This kind of result, in the modern terminology, is called \textit{coincidence
result}. The notion of cotype of a Banach space appeared in the 70's with
works of J. Hoffmann-J\o rgensen \cite{HJ}, B. Maurey \cite{Ma2}, S.
Kwapie\'{n} \cite{K22}, E. Dubinsky, A. Pe\l czy\'{n}ski, H.P. Rosenthal
\cite{DPR}, H.P. Rosenthal \cite{Ro} among others.

A Banach space $E$ has cotype $s\in\lbrack2,\infty)$ if there is a constant
$C\geq0$ so that, for all positive integer $n$ and all $x_{1},...,x_{n}$ in
$E$, we have%
\begin{equation}
\left(  \underset{i=1}{\overset{n}{%
{\displaystyle\sum}
}}\left\Vert x_{i}\right\Vert ^{s}\right)  ^{1/s}\leq C\left(  \int_{0}%
^{1}\left\Vert \underset{i=1}{\overset{n}{%
{\displaystyle\sum}
}}r_{i}\left(  t\right)  x_{i}\right\Vert ^{2}dt\right)  ^{1/2},\label{2.3}%
\end{equation}
where, for all $i$, $r_{i}$ represents the $i$-th Rademacher function. By
$\cot E$ we denote the infimum of the cotypes assumed by $E$, i.e.,%
\[
\cot E:=\inf\left\{  2\leq q<\infty;E\text{ has cotype }q\right\}  .
\]
It is important to recall that the infimum in the definition of $\cot E$ may
not be achieved by $E$. The straight relation between cotype and absolutely
summing operators can be seen in numerous works. For instance:

\begin{itemize}
\item (B. Maurey and G. Pisier \cite{pisier} ) If $\dim E=\infty$, then
\[
\cot E=\inf\left\{  r:\Pi_{r;1}(E;E)=\mathcal{L}(E;E)\right\}  .
\]

\item (M. Talagrand \cite{T2, T1}) A Banach space has cotype $s>2$ if and only
if%
\[
\Pi_{s;1}(E;E)=\mathcal{L}(E;E)
\]
and the result fails for $s=2.$

\item (B. Maurey \cite{Ma2}, L. Schwartz \cite{LS}) If $F$ is an
infinite-dimensional Banach space with cotype $s>2$, then%
\begin{equation}
\inf\left\{  r:\Pi_{r}(C(K);F)=\mathcal{L}(C(K);F)\right\}  =s \label{motiv}%
\end{equation}
and the infimum is \textit{not} attained.

\item (G. Botelho \textit{et al} \cite{z}) If $2\leq r<\cot F$ and $\dim
E=\dim F=\infty$, then
\begin{equation}
\Pi_{q,r}(E,F)=\mathcal{L}(E,F)\Rightarrow\mathcal{L}(\ell_{1},\ell_{\cot
F})=\Pi_{q,r}(\ell_{1},\ell_{\cot F}). \label{231}%
\end{equation}

\end{itemize}

In this note we will be concerned with results somewhat similar to
(\ref{motiv}) in the framework of multilinear mappings. More precisely we will
be concerned in proving the sharpness of a family of coincidence results from
the multilinear theory; we will show that, contrary to what happens in
(\ref{motiv}), the limit points (suprema and infima) will always be attained.

The nonlinear theory of absolutely summing operators was first sketched by
Pietsch in \cite{pi1} and since then it has been developed in different
directions, also with applications outside of Mathematical Analysis (see
\cite{qu}). We refer the interested reader to \cite{cha, cha2, FaJo, joed,
advances} and references therein.

Throughout this note $E_{1},\ldots,E_{n}$ and $F$ will stand for Banach spaces
over $\mathbb{K}$ $=\mathbb{R}$ or $\mathbb{C}$. If $s\in\left[
2,\infty\right)  $, the conjugate of $s$ will be represented by $s^{\ast}$ and
the class of all Banach spaces of cotype $s$ will be represented by
$\mathcal{C}^{(s)}$. By $\mathcal{L}(E_{1},\ldots,E_{n};F)$ (or $\mathcal{L}%
(^{n}E;F)$ if $E_{1}=\cdots=E_{n}=E$) we denote the Banach space of all
continuous $n$-linear mappings from $E_{1}\times\cdots\times E_{n}$ to $F$
with the usual sup norm.

If $1\leq q_{1},...,q_{n}\leq p<\infty$, an $n$-linear mapping $T\in
\mathcal{L}(E_{1},\ldots,E_{n};F)$ is multiple $(p;q_{1},...,q_{n})$-summing
if there exists a constant $C_{n}\geq0$ such that
\begin{equation}
\left(  \sum_{j_{1},\ldots,j_{n}=1}^{N}\left\Vert T(x_{j_{1}}^{(1)}%
,\ldots,x_{j_{n}}^{(n)})\right\Vert ^{p}\right)  ^{\frac{1}{p}}\leq C_{n}%
\prod_{k=1}^{n}\sup_{\varphi_{k}\in B_{E_{k}^{\ast}}}\left(  \sum_{j=1}%
^{N}\left\vert \varphi_{k}(x_{j}^{(k)})\right\vert ^{q_{k}}\right)  ^{\frac
{1}{q_{k}}} \label{lhs}%
\end{equation}
for every $N\in\mathbb{N}$ and any $x_{j_{k}}^{(k)}\in E_{k}$, with
$j_{k}=1,\ldots,N$ and $k=1,\ldots,n$. The space of all such operators is
denoted $%
{\textstyle\prod\nolimits_{m\left(  p;q_{1},...,q_{n}\right)  }}
(E_{1},...,E_{n};F)$ (or $%
{\textstyle\prod\nolimits_{m\left(  p;q_{1},...,q_{n}\right)  }}
(^{n}E;F)$ if $E_{1}=\cdots=E_{n}=E$)$.$ For the theory of multiple summing
multilinear operators and applications we refer to \cite{Acosta, defant,
defant2, matos, advances, PV, popa5} and references therein.

A famous result due to H.F. Bohnenblust and E. Hille, with applications in
harmonic analysis and analytic number theory, \cite{bh} asserts that%
\begin{equation}
\inf\left\{  r:\mathcal{L}(^{n}c_{0};\mathbb{C})=%
{\textstyle\prod\nolimits_{m\left(  r;1,...,1\right)  }}
(^{n}c_{0};\mathbb{C})\text{ }\right\}  =\frac{2n}{n+1}\label{ooi}%
\end{equation}
and this infimum is attained. The proof is very deep and specially the proof
of the sharpness of the value $\frac{2n}{n+1}$ is highly non trivial. Using a
characterization of $\mathcal{L}(c_{0};X)$ in terms of weakly summable
sequences, this result can be reformulated (see \cite{arc}) to%
\begin{equation}
\inf\left\{  r:\mathcal{L}(E_{1},...,E_{n};\mathbb{C})=%
{\textstyle\prod\nolimits_{m\left(  r;1,...,1\right)  }}
(E_{1},...,E_{n};\mathbb{C})\right\}  \leq\frac{2n}{n+1},\label{ooi22}%
\end{equation}
regardless of the infinite-dimensional complex Banach spaces $E_{1},...,E_{n}%
$. It is well-known that this result is also valid for real scalars (see
\cite{defant, ddiniz}). Similar results were investigated in \cite{PAMS2008,
REMC2010} and in special particular cases the optimality was also obtained
(see \cite[Thm. 5.14]{REMC2010}). The case $n=2$ in the Bohnenblust--Hille
inequality is the well-known Littlewood's $4/3$ inequality \cite{LLL} which
asserts that there is a positive constant $C$ such that
\[
\left(
{\displaystyle\sum\limits_{i,j=1}^{\infty}}
\left\vert A(e_{i},e_{j})\right\vert ^{\frac{4}{3}}\right)  ^{\frac{3}{4}}\leq
C\left\Vert A\right\Vert
\]
regardless of the continuous bilinear form $A$ on $c_{0}\times c_{0}$.

We will be interested in the optimality of results similar to (\ref{ooi}) and
(\ref{ooi22}), under certain cotype assumptions, for absolutely summing
multilinear operators. The notion of absolutely summing multilinear operators
is a slight variation of the notion of multiple summing operators; the
difference is that in this new approach we just sum diagonally. More
precisely, if $0<q_{1},...,q_{n},p<\infty$ and%
\[
\frac{1}{p}\leq\frac{1}{q_{1}}+\cdots+\frac{1}{q_{n}}%
\]
an $n$-linear mapping $T\in\mathcal{L}(E_{1},\ldots,E_{n};F)$ is absolutely
$(p;q_{1},...,q_{n})$-summing if there exists a constant $C_{n}\geq0$ such
that%
\[
\left(  \sum_{j=1}^{N}\left\Vert T(x_{j}^{(1)},\ldots,x_{j}^{(n)})\right\Vert
^{p}\right)  ^{\frac{1}{p}}\leq C_{n}\prod_{k=1}^{n}\sup_{\varphi_{k}\in
B_{E_{k}^{\ast}}}\left(  \sum_{j=1}^{N}\left\vert \varphi_{k}(x_{j}%
^{(k)})\right\vert ^{q_{k}}\right)  ^{\frac{1}{q_{k}}}%
\]
for every $N\in\mathbb{N}$ and any $x_{j}^{(k)}\in E_{k}$, with $j=1,\ldots,N$
and $k=1,\ldots,n$. The space of all such operators is denoted $%
{\textstyle\prod\nolimits_{as\left(  p;q_{1},...,q_{n}\right)  }}
(E_{1},...,E_{n};F)$ (and $\Pi_{as(p;q_{1},...,q_{n})}(^{n}E;F)$ if
$E_{1}=\cdots=E_{n}$). For both notions (absolutely summing multilinerar
operators and multiple summing multilinear operators) when $n=1$ we recover
the classical concept of absolutely summing linear operators (see \cite{Di}).
For absolutely summing multilinear operators a kind of version of the
Bohnenblust--Hille theorem is the well-known Defant--Voigt theorem: if
$n\geq2$ is a positive integer, then%
\begin{equation}
\inf\left\{  r:\mathcal{L}(E_{1},...,E_{n};\mathbb{K})=%
{\textstyle\prod\nolimits_{as\left(  r;1,...,1\right)  }}
(E_{1},...,E_{n};\mathbb{K})\text{ for all }E_{1},...,E_{n}\right\}
\leq1\label{dvv}%
\end{equation}
and the value $r=1$ is attained.

A result due to G. Botelho \cite{ir} asserts that, under certain cotype
assumptions, the Defant--Voigt theorem (\ref{dvv}) can be improved even with
arbitrary Banach spaces $F$ in the place of the scalar field $\mathbb{K}$: if
$n\geq2$ is a positive integer, $s\in\lbrack2,\infty)$ and $F\neq\left\{
0\right\}  $ is any Banach space, then%
\begin{equation}
\inf\left\{  r:\mathcal{L}(E_{1},...,E_{n};F)=%
{\textstyle\prod\nolimits_{as\left(  r;1,...,1\right)  }}
(E_{1},...,E_{n};F)\text{ for all }E_{j}\text{ in }\mathcal{C}^{(s)}\right\}
\leq\frac{s}{n} \label{fr}%
\end{equation}
and the value $r=\frac{s}{n}$ is attained.  Using recent results (see
\cite{Ber, Bla, Pop}) it is also simple to conclude that
\[
\sup\left\{  r:\mathcal{L}(E_{1},...,E_{n};F)=%
{\textstyle\prod\nolimits_{as\left(  1;r,...,r\right)  }}
(E_{1},...,E_{n};F)\text{ for all }E_{j}\text{ in }\mathcal{C}^{(s)}\right\}
\geq\frac{sn}{sn+s-n}\text{ }%
\]
and%
\[
\sup\left\{  r:\mathcal{L}(E_{1},...,E_{n};F)=%
{\textstyle\prod\nolimits_{as\left(  2;r,...,r\right)  }}
(E_{1},...,E_{n};F)\text{ for all }E_{j}\text{ in }\mathcal{C}^{(s)}\right\}
\geq\frac{2sn}{2sn+s-2n},
\]
with both $r=\frac{sn}{sn+s-n}$ and $r=\frac{2sn}{2sn+s-2n}$\ attained.

Several recent papers have treated similar problems involving inclusion,
coincidence results and the geometry of the Banach spaces involved (see
\cite{REMC2010, BBB, Ju, St, Pop}). In fact, families of these kind of
coincidence results have been obtained by different works and techniques.
However, contrary to the case of multiple summing operators (in which some
optimal results are known), it seems to exist no available study on the
eventual sharpness of any coincidence result for absolutely summing
multilinear operators. In this note we prove that a huge family of coincidence
results is sharp and, in particular, we show that all the above inequalities
are optimal for $n\geq s$. We also stress that, in general, the hypothesis
$n\geq s$ can not be dropped.

\section{Results\label{trh}}

Very recently the following result was essentially proved independently by
different authors (see \cite[Thm 1.4]{Ber}\ and also \cite{Bla, Pop}):

If $E_{i}$ has finite cotype $s_{i}$ for $i=1,...,n$ and $z\geq\frac{q}{n}$,
then%
\begin{equation}%
{\textstyle\prod\nolimits_{as\left(  z;q,...,q\right)  }}
(E_{1},...,E_{n};F)=%
{\textstyle\prod\nolimits_{as\left(  \frac{zqp}{nz\left(  q-p\right)
+pq};p,...,p\right)  }}
(E_{1},...,E_{n};F) \label{fim}%
\end{equation}
for all $F$ and%
\[
1\leq p\leq q<\min s_{i}^{\ast}\text{ if }s_{j}>2\text{ for some }j=1,...,n
\]
or%
\begin{equation}
1\leq p\leq q\leq2\text{ if }s_{j}=2\text{ for all }j=1,...,n. \label{ou}%
\end{equation}

In fact, in \cite[Thm 1.4]{Ber} the result is stated for $z\geq1$ instead of
$z\geq\frac{q}{n}$ but by using the H\"{o}lder inequality in its full
generality one can easily see that the above result in fact is valid for all
$z\geq\frac{q}{n}.$

In view of  (\ref{fim}), if each $E_{i}$ has finite cotype $s,$ by making
$p=1$ and
\[
\frac{zqp}{nz\left(  q-p\right)  +pq}=\frac{s}{n},
\]
we have
\[%
{\textstyle\prod\nolimits_{as\left(  \frac{s}{n};1,...,1\right)  }}
(E_{1},...,E_{n};F)=%
{\textstyle\prod\nolimits_{as\left(  \frac{sq}{qn-qsn+sn};q,...,q\right)  }}
(E_{1},...,E_{n};F).
\]
for all $1\leq q<s^{\ast}$. \textit{A fortiori},
\begin{equation}
\mathcal{L}(E_{1},...,E_{n};F)=%
{\textstyle\prod\nolimits_{as\left(  \frac{sq}{qn-qsn+sn};q,...,q\right)  }}
(E_{1},...,E_{n};F)\label{sh}%
\end{equation}
for all $1\leq q<s^{\ast}$.\ Note that when $s=2$ we do need $q<2$ (contrary
to what happens in (\ref{ou})). Alternatively, the coincidence (\ref{sh})
could also have been obtained by a direct combination of (\ref{fr}) and the
inclusion theorem for absolutely summing multilinear operators. Our main
result shows the sharpness of the above family of coincidences for all $n\geq
s.$ In the next section we show that, in general, the assumption that $n\geq
s$ can not be abstained.

\begin{theorem}
\label{t55}If $n$ $\geq s$ is a positive integer, $1\leq q<s^{\ast}$ and $F$
is an arbitrary nontrivial Banach space, then%
\begin{equation}\small {
{{{ \inf\left\{  r:\mathcal{L}(E_{1},...,E_{n};F)=%
{\textstyle\prod\nolimits_{as\left(  r;q,...,q\right)  }}
(E_{1},...,E_{n};F)\text{ for all }E_{j}\text{ in }\mathcal{C}^{(s)}\right\}
=\frac{sq}{qn-qsn+sn}} } } } \label{eqw}%
\end{equation}
and
\begin{equation}\small {
{{{ \sup\left\{  t:\mathcal{L}(E_{1},...,E_{n};F)=%
{\textstyle\prod\nolimits_{as\left(  \frac{sq}{qn-qsn+sn};t,q,...,q\right)  }}
(E_{1},...,E_{n};F)\text{ for all }E_{j}\text{ in }\mathcal{C}^{(s)}\right\}
=q.} } } } \label{pppp}%
\end{equation}

\end{theorem}

\begin{proof}
Given $\varepsilon>0,$ we will show that%
\[
\mathcal{L}(E_{1},...,E_{n};F)\neq%
{\textstyle\prod\nolimits_{as\left(  \frac{sq}{qn-qsn+sn}-\varepsilon
;q,...,q\right)  }}
(E_{1},...,E_{n};F).
\]
It is obvious that it suffices to show that there exists a $\delta>0$ so that
\[
\mathcal{L}(E_{1},...,E_{n};F)\neq%
{\textstyle\prod\nolimits_{as\left(  \frac{sq-\varepsilon}{qn-qsn+sn}%
;q,...,q\right)  }}
(E_{1},...,E_{n};F)
\]
for all $0<\varepsilon<\delta.$ We begin by showing that there is an
$S_{\varepsilon}>\frac{sq-\varepsilon}{qn-qsn+sn}$ such that
\begin{equation}%
{\textstyle\prod\nolimits_{as\left(  \frac{sq-\varepsilon}{qn-qsn+sn}%
;q,...,q\right)  }}
(E_{1},...,E_{n};F)\subset%
{\textstyle\prod\nolimits_{as\left(  S_{\varepsilon};s^{\ast},...,s^{\ast
}\right)  }}
(E_{1},...,E_{n};F). \label{fai}%
\end{equation}
Recall that the Inclusion Theorem for absolutely summing multilinear operators
(see \cite[Prop. 3.3]{te} or \cite[Prop. 3.2]{monats}) asserts that%
\[%
{\textstyle\prod\nolimits_{as\left(  q;q_{1},...,q_{n}\right)  }}
\subset%
{\textstyle\prod\nolimits_{as\left(  p;p_{1},...,p_{n}\right)  }}
\]
whenever $0<q\leq p<\infty$ and $0<q_{j}\leq p_{j}<\infty$ and
\[
\frac{1}{q_{1}}+\cdots+\frac{1}{q_{n}}-\frac{1}{q}\leq\frac{1}{p_{1}}%
+\cdots+\frac{1}{p_{n}}-\frac{1}{p}.
\]

So we just need to check that%
\begin{equation}
\frac{n}{q}-\frac{qn-qsn+sn}{sq-\varepsilon}\leq\frac{n}{s^{\ast}}-\frac
{1}{S_{\varepsilon}}. \label{tr}%
\end{equation}
If $0<\varepsilon<\delta<\min\left\{  \frac{1}{n},q\right\}  $, a direct
calculation shows that%
\[
S_{\varepsilon}>\frac{sq\left(  sq-\varepsilon\right)  }{n\varepsilon\left(
s+q-sq\right)  }%
\]
satisfies (\ref{tr}). Now we just need to choose
\[
S_{\varepsilon}>\max\left\{  \frac{sq-\varepsilon}{qn-qsn+sn},\frac{sq\left(
sq-\varepsilon\right)  }{n\varepsilon\left(  s+q-sq\right)  }\right\}  .
\]
Now choose $E_{1}=\cdots=E_{n}=\ell_{s},$ $v\neq0$ in $F$ and define
$T:E_{1}\times\cdots\times E_{n}\rightarrow F$ by
\[
T\left(  \left(  x_{j}^{(1)}\right)  _{j=1}^{\infty},...,\left(  x_{j}%
^{(n)}\right)  _{j=1}^{\infty}\right)  =\left(
{\textstyle\sum\limits_{j=1}^{\infty}}
{\textstyle\prod\limits_{k=1}^{n}}
x_{j}^{(k)}\right)  v.
\]
Note that $T$ is well-defined because $n\geq s$ (we just need to invoke
H\"{o}lder's inequality). Since
\[
\sup_{\varphi\in B_{\ell_{s}^{\ast}}}\sum_{j=1}^{N}\left\vert \varphi
(e_{j})\right\vert ^{s^{\ast}}=1
\]
and%
\[
\left\Vert T\left(  e_{j},...,e_{j}\right)  \right\Vert =\left\Vert
v\right\Vert \neq0
\]
for all $j,$ it is clear that $T$ fails to be absolutely $\left(
S_{\varepsilon};s^{\ast},...,s^{\ast}\right)  $-summing. Hence, from
(\ref{fai}) the map $T$ also fails to be absolutely $\left(  \frac
{sq-\varepsilon}{qn-qsn+sn};q,...,q\right)  $-summing.

In order to prove (\ref{pppp}) we just need to show that if $\varepsilon>0,$
then%
\[
\mathcal{L}(E_{1},...,E_{n};F)\neq%
{\textstyle\prod\nolimits_{as\left(  \frac{sq}{qn-qsn+sn};q+\varepsilon
,q,...,q\right)  }}
(E_{1},...,E_{n};F).
\]
It suffices to consider $\varepsilon>0$ so that $q+\varepsilon<s^{\ast}$. Now
we prove that that there exists an $R_{\varepsilon}>\frac{sq}{qn-qsn+sn}$ such
that
\[%
{\textstyle\prod\nolimits_{as\left(  \frac{sq}{qn-qsn+sn};q+\varepsilon
,q,...,q\right)  }}
(E_{1},...,E_{n};F)\subset%
{\textstyle\prod\nolimits_{as\left(  R_{\varepsilon};s^{\ast},...,s^{\ast
}\right)  }}
(E_{1},...,E_{n};F).
\]

By invoking again the Inclusion Theorem for absolutely summing multilinear
operators, it suffices to find a $R_{\varepsilon}$ so that
\begin{equation}
\frac{1}{q+\varepsilon}+\frac{n-1}{q}-\frac{qn-qsn+sn}{sq}\leq\frac{n}%
{s^{\ast}}-\frac{1}{R_{\varepsilon}}. \label{qe}%
\end{equation}
A straightforward calculation shows that
\[
R_{\varepsilon}>\max\left\{  \frac{q\left(  q+\varepsilon\right)
}{\varepsilon},\frac{sq}{qn-qsn+sn}\right\}
\]
satisfies (\ref{qe}). Now we follow the same idea of the final part of the
proof of the equality (\ref{eqw}).
\end{proof}

\begin{corollary}
If $s\in\left[  2,\infty\right)  $ and $n\geq s$ is a positive integer, then
{\small {{
\[
\inf\left\{  r:\mathcal{L}(E_{1},...,E_{n};F)=%
{\textstyle\prod\nolimits_{as\left(  r;1,...,1\right)  }}
(E_{1},...,E_{n};F)\text{ for all }E_{j}\text{ in }\mathcal{C}^{(s)}\right\}
=\frac{s}{n},
\]
}%
\begin{equation}
{\sup\left\{  r:\mathcal{L}(E_{1},...,E_{n};F)=%
{\textstyle\prod\nolimits_{as\left(  1;r,...,r\right)  }}
(E_{1},...,E_{n};F)\text{ for all }E_{j}\text{ in }\mathcal{C}^{(s)}\right\}
=\frac{sn}{sn+s-n},}\label{rb}%
\end{equation}%
\begin{equation}
{\sup\left\{  r:\mathcal{L}(E_{1},...,E_{n};F)=%
{\textstyle\prod\nolimits_{as\left(  2;r,...,r\right)  }}
(E_{1},...,E_{n};F)\text{ for all }E_{j}\text{ in }\mathcal{C}^{(s)}\right\}
=\frac{2sn}{2sn+s-2n},\label{sa}}%
\end{equation}
}}and the infimum and both suprema are attained.
\end{corollary}

\section{Remarks and consequences of the main result}

In general, the hypothesis $n\geq s$ from Theorem \ref{t55} can not be
weakened. In fact, if $n<s$ we can choose $F=\mathbb{K}$ and $q=1$ and so%
\[
\frac{sq}{qn-qsn+sn}=\frac{s}{n}>1.
\]
Therefore Defant--Voigt Theorem (\ref{dvv}) shows that the estimate from
(\ref{eqw}) is not sharp.

Using (\ref{rb}), (\ref{sa}) and complex interpolation (as in \cite{monats,
Bla, PAMS2008, REMC2010, Ju}) one can obtain intermediate results. It seems
not obvious that the optimality of (\ref{rb}) and (\ref{sa}) imply the
optimality of the intermediate results obtained via complex interpolation. But
a straightforward computation shows that the intermediate results obtained via
the complex interpolation method are contained in the family of coincidences
(\ref{eqw}) and so are also optimal.

More precisely, using complex interpolation and (\ref{rb}), (\ref{sa}) we
conclude that, for all $0\leq\theta\leq1$ and all complex Banach spaces $F$ we
have (over the complex scalar field)%
\begin{equation}
\sup\left\{  t:\mathcal{L}(E_{1},...,E_{n};F)=%
{\textstyle\prod\nolimits_{as\left(  \frac{2}{1+\theta};t,...,t\right)  }}
\text{ for all }E_{j}\text{ in }\mathcal{C}^{(s)}\right\}  \geq\frac
{2sn}{s\theta+2sn+s-2n}.\label{ssss}%
\end{equation}
Using a complexification argument we can extend this result to real Banach
spaces. A straightforward inspection shows that if
\[
\frac{2}{1+\theta}=\frac{sq}{qn-qsn+sn}%
\]
in (\ref{pppp}) then
\[
\frac{2sn}{s\theta+2sn+s-2n}=q.
\]
We thus conclude that (\ref{ssss}) is sharp and
{\small{
\[
\sup\left\{  t:\mathcal{L}(E_{1},...,E_{n};F)=%
{\textstyle\prod\nolimits_{as\left(  \frac{2}{1+\theta};t,...,t\right)  }}
(E_{1},...,E_{n};F)\text{ for all }E_{j}\text{ in }\mathcal{C}^{(s)}\right\}
=\frac{2sn}{s\theta+2sn+s-2n}%
\]
}}
and (obviously) the supremum is attained.

It is also interesting to compare our main result with similar results for
multiple summing operators. For example, from \cite[Corollary 5.7 + Theorem
5.14]{REMC2010} we conclude that for $s=2$ we have%
\begin{equation}
\sup\left\{  r:\mathcal{L}(E_{1},...,E_{n};F)=%
{\textstyle\prod\nolimits_{m\left(  2;r,...,r\right)  }}
(E_{1},...,E_{n};F)\text{ for all }E_{j}\text{ in }\mathcal{C}^{(2)}\right\}
=\frac{2n}{2n-1}\label{sa2}%
\end{equation}
and the $\sup$ is attained; but, for $s>2$ the information is not sharp (see
\cite[Corollary 5.9]{REMC2010}):
\begin{equation}
\sup\left\{  r:\mathcal{L}(E_{1},...,E_{n};\mathbb{K})=%
{\textstyle\prod\nolimits_{m\left(  2;r,...,r\right)  }}
(E_{1},...,E_{n};\mathbb{K})\text{ for all }E_{j}\text{ in }\mathcal{C}%
^{(s)}\right\}  \geq\frac{sn}{sn-1}\label{lk}%
\end{equation}
and it seems to be not known if $r=\frac{sn}{sn-1}$ is attained or not.

From (\ref{sa}) and (\ref{sa2}) we observe the difference between the
estimates for multiple summing and absolutely summing multilinear operators.
For example, if $n\geq s=2$ we have:%
\begin{align*}
&  \sup\left\{  r:\mathcal{L}(E_{1},...,E_{n};F)=%
{\textstyle\prod\nolimits_{m\left(  2;r,...,r\right)  }}
(E_{1},...,E_{n};F)\text{ for all }E_{j}\text{ in }\mathcal{C}^{(2)}\right\}
=\frac{2n}{2n-1}\\
&  \leq\frac{2n}{n+1}=\sup\left\{  r:\mathcal{L}(E_{1},...,E_{n};F)=%
{\textstyle\prod\nolimits_{as\left(  2;r,...,r\right)  }}
(E_{1},...,E_{n};F)\text{ for all }E_{j}\text{ in }\mathcal{C}^{(2)}\right\}
.
\end{align*}
Note that the inequality above is an equality if and only if $n=2$. So, if
$n=s=2$ then we have the following somewhat surprising equality
\begin{align}
&  \sup\left\{  r:\mathcal{L}(E_{1},E_{2};F)=%
{\textstyle\prod\nolimits_{m\left(  2;r,r\right)  }}
(E_{1},E_{2};F)\text{ for all }E_{j}\text{ in }\mathcal{C}^{(2)}\right\}
\label{bgt}\\
&  =\sup\left\{  r:\mathcal{L}(E_{1},E_{2};F)=%
{\textstyle\prod\nolimits_{as\left(  2;r,r\right)  }}
(E_{1},E_{2};F)\text{ for all }E_{j}\text{ in }\mathcal{C}^{(2)}\right\}
=\frac{4}{3}.\nonumber
\end{align}
On the other hand, for other values of $n$ and $s$ we have
\[
\frac{sn}{sn-1}<\frac{2sn}{2sn+s-2n}%
\]
and a generalization of \ (\ref{bgt}) to $s>2$ seems an interesting open
question. From (\ref{sa}) and from the natural inclusion $%
{\textstyle\prod\nolimits_{m\left(  2;r,...,r\right)  }}
\subset%
{\textstyle\prod\nolimits_{as\left(  2;r,...,r\right)  }}
$ we obtain the following complement of (\ref{lk}):

\begin{corollary}
\label{end} If $s\in(2,\infty)$ and $n\geq s$ is a positive integer, then
{\small {
\[
\sup\left\{  r:\mathcal{L}(E_{1},...,E_{n};\mathbb{K})=%
{\textstyle\prod\nolimits_{m\left(  2;r,...,r\right)  }}
(E_{1},...,E_{n};\mathbb{K})\text{ for all }E_{j}\text{ in }\mathcal{C}%
^{(s)}\right\}  \in\left[  \frac{sn}{sn-1},\frac{2sn}{2sn+s-2n}\right]  .
\]
}}
\end{corollary}

\section{Open Problems}

As we mentioned in the introduction, a famous result due to Grothendieck
asserts that every continuous linear operator from $\ell_{1}$ to any Hilbert
space $H$ is absolutely summing. In their seminal paper \cite{LP}
Lindenstrauss and Pe\l czy\'{n}ski prove that if we replace $\ell_{1}$ by
$\ell_{p}$ with $p\neq1$ then there is a continuous linear operator from
$\ell_{p}$ to $H$ which fails to be absolutely summing. In fact, the result of
Lindenstrauss and Pe\l czy\'{n}ski is quite stronger: if $E$ is an
infinite-dimensional Banach space with unconditional Schauder basis, $F$ is an
infinite-dimensional Banach space and every continuous linear operator from
$E$ to $F$ is absolutely summing then $E=\ell_{1}$ and $F$ is a Hilbert space.
The special role played by $\ell_{1}$ in the theory of absolutely summing
linear operators can be also seen in (\ref{231}) and \cite{BB, ben}.

In the nonlinear setting $\ell_{1}$ also plays a special role. For example, the following multilinear extension of Grothendieck's Theorem was recently proved in \cite{Ber}:
\[
\mathcal{L}(^{n}\ell_{1};\ell_{2})=%
{\textstyle\prod\nolimits_{as\left(  \frac{2}{n+1};1,...,1\right)  }}
(^{n}\ell_{1};\ell_{2}).
\]
Note that the case $n=1$ is precisely Grothendieck's Theorem. By using the inclusion theorem for absolutely summing multilinear operators we
conclude that%
\begin{align*}
\mathcal{L}(^{n}\ell_{1};\ell_{2})  & =%
{\textstyle\prod\nolimits_{as\left(  1;\frac{2n}{n+1},...,\frac{2n}%
{n+1}\right)  }}
(^{n}\ell_{1};\ell_{2}),\\
\mathcal{L}(^{n}\ell_{1};\ell_{2})  & =%
{\textstyle\prod\nolimits_{as\left(  2;2,...,2\right)  }}
(^{n}\ell_{1};\ell_{2})
\end{align*}
and, more generally, we have
\begin{equation}
\mathcal{L}(^{n}\ell_{1};\ell_{2})=%
{\textstyle\prod\nolimits_{as\left(  \frac{2}{\theta+1};\frac{2n}{\theta
+n},...,\frac{2n}{\theta+n}\right)  }}
(^{n}\ell_{1};\ell_{2}).\label{shha}%
\end{equation}
regardless of the $\theta\in\lbrack0,1].$

\begin{problem}
Is (\ref{shha}) sharp for some $\theta\in\lbrack0,1]?$
\end{problem}

\begin{problem}
What is the precise value of the supremum in Corollary \ref{end}?
\end{problem}

\begin{problem}
Is there an equality similar to (\ref{bgt}) for $n\geq s$ and $\left(
s,n\right)  \neq\left(  2,2\right)  $?
\end{problem}

\end{document}